\documentclass[a4paper,12pt]{article}

\usepackage[english]{babel}
\usepackage[latin1]{inputenc}
\usepackage{amsmath,amsfonts,amssymb,graphicx,latexsym}
\usepackage{enumerate}
\usepackage{nicefrac}

\newtheorem{theorem}{Theorem}
\newtheorem{lemma}[theorem]{Lemma}
\newtheorem{proposition}[theorem]{Proposition}
\newtheorem{corollary}[theorem]{Corollary}
\newtheorem{definition}[theorem]{Definition}

\newenvironment{proof}[1][Proof]{\begin{trivlist}
\item[\hskip \labelsep {\bfseries #1}]}{\hfill\qed\end{trivlist}}
\newenvironment{example}[1][Example]{\begin{trivlist}
\item[\hskip \labelsep {\bfseries #1}]}{\end{trivlist}}
\newenvironment{examples}[1][Examples]{\begin{trivlist}
\item[\hskip \labelsep {\bfseries #1}]\begin{enumerate}}{\end{enumerate}\end{trivlist}}
\newenvironment{remark}[1][Remark]{\begin{trivlist}
\item[\hskip \labelsep {\bfseries #1}]}{\end{trivlist}}
\newenvironment{notation}[1][Notation]{\begin{trivlist}
\item[\hskip \labelsep {\bfseries #1}]}{\end{trivlist}}

\newcommand{\qed}{$\Box$}

\renewcommand{\epsilon}{\varepsilon}

\newcommand{\all}{\ensuremath{\text{~for~all~}}}

\newcommand{\iso}{\simeq}
\newcommand{\id}{\mathrm{id}}
\newcommand{\unit}[1]{\ensuremath{#1^{\times}}}

\newcommand{\opp}[1]{{#1^{\mathrm{op}}}}

\newcommand{\Brcl}[2][]{\ensuremath{{[#2]}_{\mathrm{c}}^{#1}}}
\newcommand{\comp}{ \phi \colon V \times E \to E}

\newcommand{\qcomp}{ \phi \colon (V,q) \times (E,p) \to (E,p)}
\newcommand{\hcomp}{ \phi \colon (V,q) \times (E,h) \to (E,h)}
\newcommand{\usigma}{\ul{\sigma}}
\newcommand{\paircomp}{ \alpha \colon (C(A,\sigma, f), \usigma) \longrightarrow (B, \tau)}
\newcommand{\extpaircomp}{ \alpha \colon (C(P), \usigma) \longrightarrow (B, \tau)}
\newcommand{\comphom}{ \alpha \colon (C_0(V,q), \tau_0) \longrightarrow (\End_R(E), \sigma_b)}
\newcommand{\hcomphom}{ \alpha \colon (C_0(V,q), \tau_0) \longrightarrow (\End_A(E), \sigma_h)}
\newcommand{\ul}[1]{\underline{#1}}
\newcommand{\inv}[1]{{}^{\iota} \! #1}

\DeclareMathOperator{\End}{End}

\DeclareMathOperator{\Char}{char}

\DeclareMathOperator{\Br}{Br}

\DeclareMathOperator{\ind}{ind}
\DeclareMathOperator{\Brc}{Br_{c}}
\DeclareMathOperator{\lcm}{lcm}
\DeclareMathOperator{\Trd}{Trd}
\DeclareMathOperator{\Sym}{Sym}
\DeclareMathOperator{\Symd}{Symd}
\DeclareMathOperator{\Skew}{Skew}
\DeclareMathOperator{\Alt}{Alt}
\DeclareMathOperator{\Sand}{Sand}
\DeclareMathOperator{\type}{type}

\DeclareMathOperator{\mcd}{mcd}

\makeindex

\begin{document}
\title{A Generalized Composition of Quadratic Forms based on Quadratic Pairs}
\author{Roland L\"otscher}
\date{7th of October 2006}
\maketitle
\begin{abstract}
For quadratic spaces which represent $1$ there is a characterization of hermitian compositions in the
language of algebras-with-involutions using the even Clifford algebra. We
extend this notion to define a generalized composition based on quadratic pairs and
determine the degrees of minimal compositions for any given quadratic pair.
\end{abstract}
\setcounter{section}{-1}
\section{Notation}
We introduce the following notations:
For an algebra $A$ over a field $F$, we denote by \ul{A} the underlying $F$-vector-space. If $A$ is central simple $\Trd_A(x)$ denotes the reduced trace of an element $x \in A$.
For an algebra $A$ over a commutative ring $R$ we denote by $\opp{A} = \{\opp{a}\mid a \in A\}$ the opposite algebra, endowed with the same vector-space structure as $A$ but the reversed multiplication.
For a vector space $V$, the notation $T(V)$ refers to the tensor algebra on $V$. \par
If $(A,\sigma)$ is a central simple algebra with an involution of the first kind, the sets of symmetric, skew-symmetric, symmetrized and alternating elements are defined by
\begin{align*}
\Sym(A,\sigma) &= \left\{ a \in A \mid \sigma(a) = a \right\} \\
\Skew(A,\sigma) &= \left\{ a \in A \mid \sigma(a) = - a \right\} \\
\Symd(A,\sigma) &= \left\{ a + \sigma(a) \mid a \in A \right\} \\
\Alt(A,\sigma) &= \left\{ a - \sigma(a) \mid a \in A \right\}
\end{align*}
If $A$ is an $F$-algebra and $B\subset A$ a sub-algebra, the centralizer of $B$ in $A$ is denoted by $C_A(B) = \{a \in A \mid ab=ba \ \forall \ b \in B \}$. The center of $A$, $C_A(A)$ is denoted by $Z(A)$. \par
Let $A$ be Azumaya with center $Z(A) \iso F\times F$. Fix an isomorphism $\theta \colon F \times F \to Z(A)$ and let $e = \theta \big( (1,0) \big)$. We will write $A = A^{+}\times A^{-}$, where $A^{+}=eA$ and $A^{-} = (1-e)A$ are central simple $F$-algebras. \par
The standard involution on a separable quadratic extension $S/F$ is denoted by $\iota$. If $B$ is $S$-Azumaya $\inv{B}$ denotes the algebra with scalar multiplication twisted through $\iota$.
For a commutative ring $R$ the notation $\Br(R)$ refers to its Brauer group and $[A]$ indicates the class of an Azumaya $R$-algebra $A$ in the Brauer group.
\section{Introduction}
Throughout this paper we fix a base field $F$ of arbitrary characteristic. All quadratic and hermitian spaces are supposed to be finite dimensional. A $m$-dimensional quadratic space is a couple $(V,q)$, where $V$ is an $m$-dimensional vector space and $q$ is a regular quadratic form on $V$ if $m$ is even, and a semi-regular quadratic form on $V$ if $m$ is odd, see \cite{qhfor}. \par
Let $R$ be a commutative ring, $\epsilon = \pm 1$ and $A$ be an $R$-algebra with involution $a \mapsto \bar{a}$. An $\epsilon$-hermitian space over $A$ is a couple $(E,h)$, where $E$ is a faithfully projective, finitely generated $A$-right-module and $h\colon E \to A$ is a regular, with respect to the bar-involution $\epsilon$-hermitian form on $E$. For $A=R$ with the identity as involution, $\epsilon$-hermitian spaces are called symmetric bilinear spaces ($\epsilon=1$) and antisymmetric bilinear spaces ($\epsilon = -1$). \par
For a $\epsilon$-hermitian space $(E,h)$ we denote by $\sigma_h$ the adjoint involution with respect to $h$, i.e.~the involution of $\End_A(E)$ subject to the condition 
$$ h\big(f(x),y\big)=h\big(x,\sigma_h(f)(y)\big),$$ for any $x,y \in E, f \in \End_A(E)$. \par
Let now $A$ be $R$-Azumaya with an involution $a \mapsto \bar{a}$ and $E$ a faithfully projective $A$-module. Let $\tau$ be an involution (of the first or second kind) on $\End_A(E)$. By a generalized Skolem-Noether theorem, there exists an invertible $R$-module $I$, such that $\tau$ is adjoint to a nonsingular $\epsilon$-hermitian form on $E$ with values in $A \otimes I$, see \cite{Aamp}. \par
For $A$ as above there exists a faithfully flat $R$-algebra $S$ such that $A\otimes S \iso \End_S(E)$, where $E$ is a faithfully projective $S$-module. Assume that $\tau$ is of the first kind. The induced involution $\tau \otimes \id_S$ is then adjoint to an $\epsilon$-symmetric bilinear form $b$ on $E$. In convention with \cite{boi} we call $\tau$ of \emph{symplectic type} if $b$ is alternating, i.e. $b(y,y)=0$ for $y \in E$, and of \emph{orthogonal type} otherwise.

Let now $(V,q)$ and $(E,p)$ be quadratic spaces.
A bilinear map $$\comp$$ satisfying 
\begin{equation}
\label{eq-qcomp}
p(\phi(x,y)) = q(x)p(y), \ \forall x, \forall y.
\end{equation}
is called a
 \emph{quadratic composition of $ (V,q) $ with $ (E,p)$}.
Let $(E,h)$ be an $\epsilon$-hermitian space over an $F$-algebra $A$ with involution. A map $\comp$ which is $R$-linear in the first variable and $A$-linear in the second variable is called \emph{$\epsilon$-hermitian composition} of $(V,q)$ with $(E,h)$, if it satisfies the equation
\begin{equation}
\label{eq-hcomp}
h\big(\phi(x,y_1), \phi(x,y_2)\big) = q(x) h(y_1,y_2), \quad \forall x \in V, \forall y_1,y_2 \in E. 
\end{equation}

A quadratic composition $\comp$ of spaces $(V,q)$ and $(E,p)$ induces a composition of $(V,q)$ with the polar form of $p$. Equation \eqref{eq-hcomp} with $h = b_p$ results from linearizing equation \eqref{eq-qcomp} in the second variable. If $\Char F \neq 2$, then the two equations are really equivalent. For fields of characteristic $2$ (and more generally for rings, in which $2$ is a zero-divisor) the situation is more complicated. There are examples, in which a composition of a quadratic space with an even symmetric bilinear form is not induced by a quadratic composition. However, that only happens for quadratic spaces of dimension $\leq 5$, see the following theorem from \cite[Theorem 4]{loetscher1}:

\begin{theorem}
\label{th1}
Let $\comp$ be a composition of a quadratic space $(V,q)$ with a symmetric bilinear space $(E,b)$. Assume  $\dim V \geq 6$ and the existence of $z \in V$ with $q(z)=1$. If $\dim V$ is even, then there exists a quadratic form $p$ on $E$ with polar $b$ and a quadratic composition of $(V,q)$ with $(E,p)$. If $\dim V$ is odd, then the same statement holds, if we assume that $z$ is contained in a regular subspace of $(V,q)$.
\end{theorem}

Compositions of quadratic spaces with $\epsilon$-hermitian spaces can be characterized in terms of algebras-with-involution. For a quadratic space $(V,q)$ let $C_0=C_0(V,q)$ denote its even Clifford algebra with canonical involution $\tau_0$. For the proof of the next theorem, see \cite[Theorem 3]{loetscher1}.
\begin{theorem}
\label{char-th}
Let $(V,q)$ be a quadratic space containing an element $z \in V$ with $q(z) = 1$ and let $(E,h)$ be a $\epsilon$-hermitian space. There exists a composition $\comp$ of $(V,q)$ with $(E,h)$ iff there exists a homomorphism $\hcomphom$ of algebras-with-involution.
\end{theorem}

\begin{example}
Assume $\Char F \neq 2$ and let $$(V,q) = (F^5, \langle 1,-a,-b,-1,1 \rangle)$$ for some $a,b \in F^{\times}$. We take $z = (1,0,0,0,0)$ and decompose $(V,q)$ as $(Fz \oplus V', \langle 1 \rangle  \perp -q')$. Sending $x  \in V'$ to $zx \in C_0(V,q)$ yields, by the universal property of $C(V,q)$, an isomorphism $$(C_0(V,q),\tau_0) \iso (C(V',q'),\sigma),$$ where $\sigma$ is the standard involution on $C(V',q')$. Let $Q = (a,b)_F$ be the quaternions with generators $i, j$ and relations $i^2 =a, j^2=b, ij+ji=0$ and let $\{e_i\}_{i=1 \dots 4}$ be the canonical basis of $V' = F^4$. 
We get an isomorphism $C(V',q') \iso M_2(Q)$ by sending
$$e_1 \mapsto \left( \begin{array}{cc} i & 0 \\ 0 & i \end{array} \right), 
e_2 \mapsto \left( \begin{array}{cc} j & 0 \\ 0 & j \end{array} \right), 
e_3 \mapsto \left( \begin{array}{cc} 0 & 1 \\ 1 & 0 \end{array} \right) \ \text{and} \ 
e_4 \mapsto \left( \begin{array}{cc} 0 & 1 \\ -1 & 0 \end{array} \right).$$
Under that isomorphism the canonical involution on $C(V',q')$ corresponds to the involution 
$$\left( \begin{array}{cc} m_{11} & m_{12} \\ m_{21} & m_{22} \end{array} \right) \mapsto \left( \begin{array}{cc} \widetilde{m_{22}} & - \widetilde{m_{12}} \\ -\widetilde{m_{21}} & \widetilde{m_{11}} \end{array} \right),$$ 
where $\tilde{a} = k \bar{a} k^{-1}$ (for $a \in Q$ with $k=ij$) is the involution of $Q$ which fixes $i,j$ and sends $k$ to $-k$. The above involution on $M_2(Q) \iso \End_Q(\ul{Q}\oplus \ul{Q})$ is adjoint to a $\epsilon$-hermitian form with respect to any fixed involution on $Q$. For the involution $a \mapsto \tilde{a}$ on $Q$ the above involution is adjoint to the anti-hermitian form $$h_1\big((y_1,y_2),(y_1',y_2')\big) = \tilde{y_1} y_2' - \tilde{y_2} y_1'$$ on $\ul{Q} \oplus \ul{Q}$. For the canonical involution $a \mapsto \bar{a}$ on $Q$ the above involution is adjoint to the hermitian form $$h_2\big((y_1,y_2),(y_1',y_2')\big) = \bar{y_1} k y_2' - \bar{y_2} k y_1'.$$
The resulting homomorphism 
$$\alpha \colon (C_0(V,q), \tau_0) \stackrel{\iso}\to (\End_Q(\ul{Q} \oplus \ul{Q}), \sigma_{h_i})$$
corresponds to an anti-hermitian composition of $(V,q)$ with $(\ul{Q} \oplus \ul{Q},h_1)$ (with respect to the tilde-involution) and to a hermitian composition of $(V,q)$ with $(\ul{Q} \oplus \ul{Q}, h_2)$ (with respect to the bar-involution), respectively. It is obtained by $\phi(x,y) = \alpha(zx)(y)$ for $x = (x_0,x_1,x_2,x_3,x_4) \in V$ and $y = (y_1,y_2) \in \ul{Q} \oplus \ul{Q}$. Explicitly, we have $$\phi(x,y) = \left( \begin{array}{ccc} (x_0+x_1 i + x_2 j) \cdot y_1 &+& (x_3+x_4) \cdot y_2 \\ (x_3 - x_4) \cdot y_1 &+& (x_0 - x_1 i - x_2 j) \cdot y_2 \end{array} \right)$$
and one may check, that $h_i(\phi(x,y),\phi(x,y')) = q(x)h_i(y,y')$ for $i=1,2$.
\end{example}

In the more general setting of quadratic forms over a commutative ring~$R$ Zger \cite{zueger2} has studied in detail $\epsilon$-hermitian compositions over $R$, over a separable quadratic extension $S/R$ or over a quaternion algebra $Q/R$. In the present paper, we only consider quadratic spaces over fields. However, we consider hermitian and anti-hermitian spaces over a large class of central simple $F$-algebras and over $S$-Azumaya algebras, where $S/F$ is a separable quadratic extension. Moreover, we generalize on the side of $(V,q)$, which we replace by a so-called \emph{quadratic pair}. \par
Quadratic pairs were introduced in \cite{boi}. The notion of an even Clifford algebras is generalized in their setting. We briefly recall the definition of a quadratic pair and its Clifford algebra, following \cite{boi}. \par
Let $A$ be a central simple $F$-algebra of degree $n$. A quadratic pair on $A$ is a couple $(\sigma,f)$, where $\sigma$ is an involution of the first kind, that is of orthogonal type if $\Char F \neq 2$ and of symplectic type if $\Char F = 2$, respectively, and $f \colon \Sym(A, \sigma) \to F$ is a $F$-linear map subject to the condition: 
\begin{equation}
\label{eq-qpair}
f(x + \sigma(x)) = \Trd_A(x) \all x \in A.
\end{equation}

If $A$ is not known from the context, we shall write $(A,\sigma,f)$ for a quadratic pair $(\sigma,f)$ on $A$. \par
If $\Char F \neq 2$, the map $f$ is uniquely determined by \eqref{eq-qpair}, $f(x) = \frac{1}{2} \Trd_A(x)$. If $\Char F = 2$, the map $f$ is only determined on $\Symd(A,\sigma) \subsetneq \Sym(A, \sigma)$ and in general there exist several maps, with which the involution $\sigma$ forms a quadratic pair on $A$. \par
In any case, for a given quadratic pair $(A,\sigma,f)$ there exists an element $\ell \in A$ with 
$$f(s) = \Trd_A(\ell s) \all s \in \Sym(A, \sigma) \quad \text{and}~\ell + \sigma(\ell)=1,$$
The element $\ell$ is unique up to additivity of an element in $\Alt(A, \sigma)$.
If $\Char F \neq 2$, it can be taken to be $\ell = \frac{1}{2}$. \par
The (generalized even) Clifford algebra is defined as a quotient of the tensor algebra $T(\ul{A})$ of the underlying vector space $\ul{A}$ of $A$:
$$C(A,\sigma,f) = \frac{T(\ul{A})}{J_1(\sigma,f)+J_2(\sigma,f)},$$
where the ideals $J_1(\sigma,f)$, $J_2(\sigma,f) \subset T(\ul{A})$ are given as follows: 
\begin{itemize}
\item $J_1(\sigma,f)$ is generated by all elements of the form $s-f(s)$, for $s \in \ul{A}$ with $\sigma(s)=s$. 
\item $J_2(\sigma,f)$ is generated by all elements of the form $u - \Sand(u)(\ell)$, for $u \in \ul{A} \otimes \ul{A}$ with $\Sand(u)(x) = \Sand(u)(\sigma(x)) \all x$. 
\end{itemize}
The Clifford algebra is equipped with an involution $\usigma$, which is induced by the involution $\sigma$ on $A$. By construction:
$$\usigma(a_1 \otimes \dotsm \otimes a_k) = \sigma(a_k) \otimes \dotsm \otimes \sigma(a_1) \all k \in \mathbb{N}, a_1,\dotsc, a_k \in \ul{A}, $$ considered as elements of the Clifford algebra. \par
Quadratic pairs on trivial algebras, i.e. of the form $A= \End_F(V)$, can be identified with quadratic forms modulo a factor in $\unit{F}$ as follows:
For a regular quadratic space $(V,q)$ consider the map
$$V \otimes V \to \End_F(V), \quad \varphi_q(v\otimes w)(x) = v b_q(w,x).$$
The map $\varphi_q$ is a linear bijection and induces an isomorphism of algebras-with-involution:
$$ \varphi_q \colon (V \otimes V, \sigma_{\mathrm{sw}}) \xrightarrow{\iso} (\End_F(V), \sigma_q),$$
where $\sigma_{\mathrm{sw}}(v\otimes w) = w \otimes v$ and the multiplication on $V\otimes V$ is suitably defined. \par
There exists a unique linear map $f_q \colon \Sym(\End_F(V), \sigma_q) \to F$ subject to the condition:
$$ f_q \circ \varphi_q(v \otimes v) = q(v) \all x,y \in V.$$
The couple $(\sigma_q, f_q)$ is a quadratic pair on the algebra $\End_F(V)$. Conversely, every quadratic pair $(\sigma,f)$ on $\End_F(V)$ is of the form $(\sigma_q, f_q)$ for a regular quadratic form $q$, which is unique up to a factor in $\unit{F}$. \par
The Clifford algebra of a quadratic pair $(\sigma_q, f_q)$ on $\End_F(V)$ is isomorphic to the even Clifford algebra of the quadratic space $(V,q)$ by a canonical isomorphism, under which the involution $\usigma$ corresponds to the canonical involution $\tau_0$ on $C_0(V,q)$. \par
The Clifford algebra construction commutes with scalar extension. If extending scalars to a splitting field $L$, a quadratic pair becomes a quadratic form up to a factor in $\unit{L}$ and its Clifford algebra becomes the usual even Clifford algebra of a quadratic space, which contributes to derive the basic structure properties of the Clifford algebra of a quadratic pair. The following two theorems are taken from \cite[Theorems 8.10, 9.14, 9.16; 8.12]{boi}. 
\begin{theorem}[Structure Theorem]
\label{structure-th}
Let $(A,\sigma,f)$ be a quadratic pair of degree $n$ and let $C=C(A,\sigma,f)$ be its Clifford algebra.
\begin{enumerate}
\item $n = 2k$: $Z = Z(C)$ is a separable quadratic $F$-Algebra and $C$ is $Z$-Azumaya of degree $2^{k-1}$. 
	If $Z \iso F \times F$, then $C \iso C^{+} \times C^{-}$ for Azumaya $F$-algebras $C^{\pm}$ of degree $2^{k-1}$ and we have
	$[C^{+}][C^{-}][A] = 1$ in $\Br(F)$ if $k$ is even, and $[C^{+}][C^{-}] = 1$ if $k$ is odd, respectively.
\item $n = 2k+1$: $(A, \sigma, f) \iso (\End_F(V), \sigma_q, f_q)$ for a regular quadratic space $(V,q)$; we have $C \iso C_0(V,q)$
	which is $F$-Azumaya. 
\end{enumerate}
\end{theorem}

\begin{theorem}
Let $(A,\sigma,f)$ be a quadratic pair of degree $n$ or a semi-regular quadratic space of dimension $n$ (if $\Char F=2$ and $n$ is odd) and let $\usigma$ be the canonical involution of the corresponding (generalized even) Clifford algebra.
\begin{enumerate}
\item $n = 2k$: The involution $\usigma$ is unitary if $k$ is odd, 
and orthogonal if $\Char F \neq 2,\ k \equiv 0 \pod 4$, 
and symplectic if $\Char F = 2$ or $k \equiv 2 \pod 4$, respectively.
\item $n = 2k+1$: The involution $\usigma$ is orthogonal if $\Char F \neq 2,\ n \equiv \pm 1 \pod 8$, 
and symplectic if $\Char F = 2,\ n > 1$ or $n \equiv \pm 3 \pod 8$, respectively. 
The only exception is $n=1,\ \Char F = 2$, where $\usigma$ is orthogonal.
\end{enumerate}
\end{theorem}

\section{Generalization of the Composition Problem}
We replace the usual Brauer group $\Br(R)$ defined for a commutative ring $R$ by restricting the equivalence classes to Azumaya algebras of \emph{constant rank}. 
\begin{definition}
\label{Brauer-Def}
Let $R$ be an arbitrary commutative ring. Recall that two $R$-Azumaya algebras $A$ and $B$ are called \emph{Brauer-equivalent}, if there exist faithfully projective $R$-modules $P_1$, $P_2$ with $A \otimes \End_R(P_1) \iso B~\otimes~\End_R(P_2)$.
The quotient of the monoid of Azumaya $R$-algebras of constant rank modulo Brauer-equivalence is called \emph{constant Brauer group}, denoted by $\Brc(R)$. The equivalence class of an Azumaya $R$-algebra $A$ of constant rank is denoted by $\Brcl{A}$.
\end{definition}

Clearly, the above definition yields nothing new for rings with connected spectrum, in particular for fields. In general the inclusion $$j \colon \Brc(R) \hookrightarrow \Br(R), \quad \Brcl{A} \mapsto [A].$$ is actually an isomorphism:
\begin{lemma}
Any equivalence class $c \in \Br(R)$ contains an Azumaya-algebra of constant rank.
\end{lemma}
\begin{proof}
Let $c =[A]$ for an Azumaya-algebra $A$ of not necessarily constant rank. We decompose $R$ into a finite product $R=R_1 \times \dotsm \times R_l$ and write $A$ as $A=A_1 \times \dotsm \times A_l$ such that $A_i$ is $R_i$-Azumaya of constant rank $n_i^2$. Let $n=n_1 \cdot \dotsm \cdot n_l$ and $P_i= R_i^{n/n_i}$, $i=1,\dotsc,l$. Then $P_1 \oplus \dotsb \oplus P_l$ is a faithfully projective $R$-module and the Azumaya $R$-algebra $A'=A \otimes \End_R(P)$ satisfies $[A']=[A]$ and has constant rank $n^2$. \par
\end{proof}

\begin{definition}
Let $(A, \sigma, f)$ be a quadratic pair over a field $F$ and let $B$ be an $F$-algebra with involution $\tau$.
We call a homomorphism $$\paircomp$$ of algebras-with-involutions a composition of type $(c,t) \in \Br(F)\times \mu_2(F)$ if $B$ is central simple over $F$ and $\tau$ is of the the first kind, with $c = [B]$ and $t=-1$ if $\tau$ is symplectic, otherwise $t=1$. \par
Let $S$ be a separable quadratic extension of $F$ and assume $B$ to be $S$-Azumaya of constant rank. We call $\paircomp$ a composition of type $(c',0) \in \Brc(S) \times \{0\}$, if $\tau$ is an involution of the second kind (with $\tau|F = \id_F$), where $c' = \Brcl{B}$. \par
In both cases we talk about compositions of the quadratic pair $(A,\sigma,f)$ and we call the couple $(c,t)$ or $(c',0)$, respectively, the type of the composition, denoted by $\type \alpha$. 
\end{definition}
The motivation for the preceding definition comes from Theorem \ref{char-th}. Recall from the introduction that involutions on Azumaya algebras with involution are adjoint to $\epsilon$-hermitian forms. Thus, we have:
\begin{proposition}
Let $(\sigma_q, f_q)$ be a quadratic pair on $A = \End_F(V)$. Scaling $q$, we may assume, that it represents $1$. Then compositions of
$(\sigma_q, f_q)$ correspond to compositions of $(V,q)$ with nonsingular $\epsilon$-hermitian spaces $\hcomp$, where $(B, \tau) \iso (\End_A(E),\sigma_h)$ 
and $h$ is $\epsilon$-hermitian with respect to some involution on $A$ which restricts to the same involution as $\tau$ on $Z(B)=Z(D)$.
\end{proposition}

Moreover, from Theorem \ref{th1} we get that quadratic compositions of spaces $(V,q)$ and $(E,p)$ correspond to compositions of the quadratic pair $(\sigma_q, f_q)$ with an orthogonal (resp.~symplectic, if $\Char F = 2$) involution on the trivial algebra $B=\End_F(E)$, assuming $\dim V \geq 6$. \par
If $(\sigma,f)$ is a quadratic pair on a non-trivial central simple algebra $A$, we may extend scalars to a splitting field $L$ to get a $\epsilon$-hermitian composition of a quadratic space (over $L$) again. If choosing $L$ large enough such that $\Brcl{B \otimes L} = 1 \in \Brc(Z(B)\otimes L)$, scalar extension yields compositions of quadratic spaces with $\epsilon$-symmetric bilinear forms if $Z(B)=F$, and with $\epsilon$-hermitian forms with respect to the standard involution of $Z(B) \otimes L$ over $L$ if $Z(B)/F$ is a separable quadratic extension, respectively.

\begin{examples}
\item If $\paircomp$ is a composition of type $(c,t)$ and $(B',\tau')$ is a central simple algebra over $F$ with involution of the first kind, then $C(A,\sigma,f) \to (B,\tau) \hookrightarrow (B\otimes B', \tau \otimes \tau')$ is a composition of $(A,\sigma,f)$ with $(B\otimes B', \tau \otimes \tau')$ of type $(c\cdot c', t \cdot t')$.
\item Let $(Q_1,\gamma_1)$ and $(Q_2,\gamma_2)$ be quaternion algebras over $F$ with standard involutions. 
There exists a canonical quadratic pair $(A,\sigma,f)$ with $A = Q \otimes Q_2$, $\sigma = \gamma_1 \otimes \gamma_2$ and its Clifford algebra is $(Q_1,\gamma_1) \times (Q_2, \gamma_2)$, see \cite[Example (8.19)]{boi}. Projection to the first factor yields a homomorphism: $(C(A,\sigma,f), \usigma) \to (Q_1,\gamma_1)$, hence a composition of $(A,\sigma,f)$ with $(Q_1,\gamma_1)$. Including $(Q_1,\gamma_1)$ in $(A,\sigma)$ gives a composition of $(A,\sigma,f)$ with $(A,\sigma)$. 
\end{examples}

\begin{remark}
For trivial algebras $A=\End_F(V)$ there exists a composition of $(A,\sigma_q,f_q)$ with $(A,\sigma_q)$ whenever $(V,\lambda \cdot q)$ is a unitary composition algebra for some $\lambda \in \unit{F}$, i.e.~a quadratic space (representing $1$) that admits a composition with itself. If $\dim V \geq 6$ or $\Char F=2$ the converse is also true, which follows from Theorem \ref{th1}. \par
It would be interesting to generalize the notion of composition of two quadratic spaces to quadratic pairs. A \emph{composition of two quadratic pairs} $(A,\sigma,f)$ and $(B,\tau,g)$ should be defined such that it induces a composition of $(A,\sigma,f)$ with $(B,\tau)$ like introduced in the present paper. For fields of characteristic $2$ the notions should show to be equivalent (presumably also if $\deg A \geq 6$, like in Theorem \ref{th1}). Also it should be stable under scalar extension, so that a composition of two quadratic pairs yields a composition of quadratic spaces over a common splitting field. Interestingly, for trivial algebras a composition $\qcomp$ induces a bilinear map $\psi \colon \Sym(\End_F(V),\sigma_q) \times \Sym(\End_F(E),\sigma_p) \to \Sym(\End_F(E),\sigma_p)$ such that $f_q(s) f_p(t) = f_p(\psi(s,t))$. It is given on generators $s=\varphi_q(x\otimes x)$, $t=\varphi_p(y\otimes y)$ as $\psi(s,t) = \varphi_p(\phi(x,y) \otimes \phi(x,y))$. Note that $\phi$ can be expressed through the composition homomorphism $\comphom$ if choosing an element $z \in V$ with $q(z)=1$.
Therefore we suggest to define a composition of two quadratic pairs $(A,\sigma,f)$ and $(B,\tau,g)$ via a bilinear map $\psi\colon \Sym(A,\sigma) \times \Sym(B,\tau) \to \Sym(B,\tau)$ subject to the condition $f(s)g(t) = g(\psi(s,t))$ and some further restrictions. More restrictions are needful since the bilinear map $\psi(s,t) = f(s)t$ should not be allowed. Unfortunately we did not manage to define $\psi$ in an explicit way nor to find suitable restrictions. \par
A \emph{composition pair} $(A,\sigma,f)$ could then be defined as a quadratic pair which admits a composition with itself. Much of the classical results follow directly from the existence of a composition $(A,\sigma,f) \to (A,\sigma)$. For example $\deg A$ must be $1,2,4$ or $8$. If $\deg A = 8$ it is necessary that the center of the (even) Clifford algebra splits and one factors is trivial (hence the other factor is isomorphic to $A$). See \cite[Corollary 2]{elduque} for the situation on trivial algebras, where those two conditions are proved to be sufficient as well.
\end{remark}

In characteristic $2$ there do not exist quadratic pairs of odd degree. In that case we consider semi-regular quadratic spaces instead of quadratic pairs. Let us introduce the following notation:
\begin{notation}
We denote by the letter $P$ a quadratic pair $(A, \sigma, f)$ or (in odd dimension, $\Char F = 2$) a semi-regular quadratic space $(V,q)$. We call
$P$ an \emph{extended quadratic pair}. Furthermore we denote by $(C(P),\usigma)$ the (even) Clifford algebra $C(A,\sigma, f)$ and $C_0(V,q)$,
respectively, with canonical involution. We write $\deg P$ for $\deg A$ and $\dim V$, respectively.
\end{notation}

\begin{remark}
For any extended quadratic pair $P$ and any given type $(c,t) \in \Br(F) \times \mu_2(F)$ with $c^2 = 1 \in \Br(F)$ there do exist compositions. A composition of such type can be constructed as follows: 
Consider the embedding $C(P)$ in $C(P) \otimes \opp{C(P)} \iso \End_Z(E) \subset \End_F(E)$. By Lemma \ref{ext} and Lemma \ref{fxfext} the involution $\usigma$ on $C(P)$ can be extended to an involution of the first kind $\rho$ on $\End_F(E)$. 
Secondly, let $D$ be central simple of Brauer class $c$. Since $c^2=1 \in \Br(F)$, there exist an involution of the first kind $\tau$ on $D$. We take the tensor product with $(D,\tau)$. The type of $\rho \otimes \tau$ may not be as required. Taking the tensor product with $M_2(F)$ equipped with the standard involution finds a remedy for that. \par
The same holds for compositions with unitary involutions: Let $S/F$ be a separable quadratic extension and $c' \in \Brc(S)$ satisfying $N_{S/F}(c')=1 \in \Br(F)$, i.e.~there exists a unitary involution $\tau'$ on $D'$. Taking the tensor product $(C(P), \usigma) \subset (\End_F(E),\rho) \subset (\End_F(E) \otimes D', \rho \otimes \tau')$ yields a composition of $P$ of type $(c',0)$.
Thus, one always gets a composition of given type. The difficulty is just to find and to reach the minimal degree.
\end{remark}

\section{Minimal Compositions}
We say, that a composition $\extpaircomp$ is \emph{minimal}, if there exists no composition $ \beta \colon (C(P), \usigma)
\longrightarrow (B', \tau')$ with $\type \alpha = \type \beta$ and $\deg B' < \deg B$. Note, that in contrast to most other works on
compositions we fix $P$ itself, not only the degree of $P$. We are interested to find a good composition for every $P$ (and every suitable type of composition). By the minimal
composition degree of given composition type $\alpha$ for $P$ we mean $\deg B'$ for any minimal composition $ \beta \colon (C(P), \usigma)
\longrightarrow (B', \tau')$ of that type. We denote the minimal composition
degree by $\mcd(P,c,t)$ for compositions of type $(c,t)$. \par
We will express the minimal composition degrees $\mcd(P, \alpha)$ with the help of a metric on the constant Brauer group. In order to introduce the metric, we need some preparations.
For an Azumaya $R$-algebra of constant rank, the minimum degree of $B$ taken over all Azumaya $R$-algebras $B$ of constant rank with $\Brcl{B}=\Brcl{A}$ is called \emph{Index} of $A$, denoted by $\ind A$. We have $\ind M_r(D) = r$ if $D$ is a division algebra with center $F$. If $A$ is Azumaya over $R=F\times F$, then $A$ is of the form $A_1\times A_2$ with $A_1$, $A_2$ central simple over $F$ and $\deg A_1=\deg A_2 = \deg A$. It is easy to see then, that $\ind A = \lcm(\ind A_1, \ind A_2)$. In both cases the index of $A$ divides the degree of $A$. \par
\begin{lemma} Let $R$ be a commutative ring and $\Brc(R)$ the constant Brauer group. The map
$$d: \Brc(R) \times \Brc(R) \rightarrow \mathbb{R}, \quad d(\Brcl{B_1},\Brcl{B_2}) = \log_2 \ind(B_1 \otimes \opp{B_2})$$ is well-defined
and satisfies the axioms of a metric. We omit the brackets and write e.g. $d(B_1,B_2)$ instead of
$d(\Brcl{B_1},\Brcl{B_2})$ for short. The above metric has the following properties: $$d(A\otimes B, A \otimes C) = d(B,C) \quad
\text{and} \quad d(\opp{B}, \opp{C}) = d(B,C) $$ for $R$-Azumaya-algebras $A, B, C$ of constant rank.
\end{lemma}
\begin{proof}
By construction, the index of an Azumaya-algebra of constant rank only depends on its class in the constant Brauer group.
Using the properties $1)~\Brcl[-1]{A} = \Brcl{\opp{A}},\ 2)~\ind A = \ind \opp{A}$ und $3)~\ind(A_1 \otimes A_2) \leq
\ind(A_1) \ind(A_2)$ the assertions can easily be shown. 
\end{proof}

\begin{remark}
If $R=F$, the induced metric on the subgroup of classes of central simple algebras with involutions of the first kind maps to $\mathbb{N}$. \end{remark}

The relevance of the metric to solve the composition problem shows up in the following two lemmas:
\begin{lemma}
\label{dbound}
Let $C$ be $R$-Azumaya of constant rank. The minimal rank of an Azumaya $R$-algebra $A$ of constant rank of given class $\Brcl{B} \in \Brc(R)$ for which there exists a homomorphism $f \colon C \to A$ of $R$-algebras is exactly $\deg A = \deg C \cdot 2^{d(B,C)}$. 
If $R = F$ or $R = F \times F$ then the degree of any (not necessarily minimal) such $A$ is a multiple of $\deg C \cdot 2^{d(B,C)}$.
\end{lemma}
\begin{proof}
For a homomorphism $f \colon C \to A$ of Azumaya algebras of constant rank we may write $A \iso C \otimes G$, where $G=C_A(f(C))$ is Azumaya of constant rank with $\Brcl{G} = \Brcl{B \otimes \opp{C}}$. 
Thus, the inequality $\deg A = \deg C \deg G \geq \deg C \ind G = \deg C \ind (B \otimes \opp{C}) = \deg C \cdot 2^{d(B,C)}$ holds. 
If $R = F$ or $R = F \times F$ then the degree of $G$ is a multiple of the index of $G$. 
Furthermore, there exists an Azumaya algebra $G$ with $2^{d(B,C)} = \deg G$ and $\Brcl{B \otimes \opp{C}} = \Brcl{G}$. Let $f \colon C \hookrightarrow A = C \otimes G$ be the
inclusion. We have $\deg A = \deg C \deg G = \deg C \cdot 2^{d(B,C)}$ and $\Brcl{A} = \Brcl{C} \Brcl{G} = \Brcl{C} \Brcl{B \otimes \opp{C}} = \Brcl{B}$.
\end{proof}

\begin{lemma}
\label{cbound}
Let $F$ be a field and $Z,S/F$ separable extensions of degree $1$ or $2$. Let $B$ and $C$ be two Azumaya algebras with center $S$ and $Z$, respectively. If there exists a homomorphism of $F$-algebras $\alpha \colon C\to B$, then the degree of $B$ is of the following form:
\begin{enumerate}
\item \label{c1} If $Z=F$ then $\deg B = n \deg C \cdot 2^{d(B,C\otimes S)}$ with $n \geq 1$.
\item \label{c2} If $Z \iso F \times F$ and $\alpha$ is not injective, then $\deg B = n \deg C \cdot 2^{d(B,C^{\pm} \otimes S)}$ with $n \geq 1$.
\item \label{c3} If $Z \iso F \times F$ and $S$ is a field, then $\deg B = \deg C \big(n_1 2^{d(B,C^{+} \otimes S)} + n_2 2^{d(B,C^{-} \otimes S)}\big)$ with $n_1+n_2 \geq 1$. Moreover, if $\alpha$ is injective and $S=F$ then $n_1, n_2 \geq 1$.
\item \label{c4} If $Z$ and $S$ are isomorphic quadratic field extensions of $F$, then $\deg B = \deg C \big(n_1 2^{d(B,C)} + 
n_2 2^{d(\inv{B},C)} \big)$ with $n_1 + n_2 \geq 1$. 
\item \label{c5} If $Z \iso F \times F \iso S$, then $\deg B = \deg C \big(n_1 2^{d(B^{+},C^{+})} + n_2 2^{d(B^{+},C^{-})}\big) = \deg C \big(m_1 2^{d(B^{-},C^{+})} + m_2 2^{d(B^{-},C^{-})}\big)$ with $n_1 + n_2 \geq 1, m_1 + m_2 \geq 1$.
\item \label{c6} If $Z$ is a quadratic field extensions of $F$ and $Z \not \iso S$, then $\deg B = n \deg C \cdot 2^{d(B\otimes Z,C \otimes S)+1}$ with $n \geq 1$.
\end{enumerate}
\end{lemma}
\begin{proof}
\begin{enumerate}
\item The homomorphism $\alpha\colon C \to B$ together with the inclusion $S \hookrightarrow B$ induces a homomorphism of $S$-algebras $C\otimes S \to B$, thus the claim follows from the preceding lemma.
\item Since $\alpha \colon C^{+} \times C^{-} \to B$ is non-injective, its kernel must be one of its nontrivial ideals $C^{+} \times \{0\}$ or $\{0\} \times C^{-}$. Thus $\alpha$ factors through $C^{-}$ or $C^{+}$ and then the preceding argument shows the claim.
\item The homomorphism $C \otimes S \to B$ from \ref{c1}.~together with the identity on $\opp{B}$ induces a $S$-linear homomorphism $\gamma \colon C^{+} \otimes \opp{B} \times C^{-} \otimes \opp{B} \to B \otimes_S \opp{B} \iso \End_S(E)$, where $E$ is the underlying $S$-vector space of $B$. Let $E_1=\gamma (1,0) E, E_2 = \gamma(0,1) E$. Then $\gamma$ restricts to homomorphisms $C^{+} \otimes \opp{B} \to \End_S(E_1)$ and $C^{-} \otimes \opp{B} \to \End_S(E_2)$. Observe that $\dim E_1, \dim E_2 > 0$ if $\alpha$ is injective and $S=F$ (and therewith $\gamma$ is injective). Since $E = E_1 \oplus E_2$, in particular $\dim E = \dim E_1 + \dim E_2$, the claim follows with \ref{c1}.
\item As above we get a homomorphism $C \otimes_F \opp{B} \to \End_S(E)$. Let $K=Z \iso S$. The map $C \otimes_F \opp{B} \to C\otimes_K \opp{B} \times C \otimes_K \inv{ \opp{B} }$ given by $c \otimes \opp{b} \mapsto (c \otimes_K \opp{b}, c \otimes_K  \inv{\opp{b}})$ is an isomorphism  $F$-algebras, yielding a homomorphism $\gamma \colon C\otimes_K \opp{B} \times C \otimes_K \inv{\opp{B}} \to \End_S(E)$, which is $S$-linear, if we view $E$ as a $S$-vector space through $\gamma$. Now the claim follows from the argument in \ref{c3}.
\item The homomorphism $C \to B = B^{+}\times B^{-}$ yields homomorphisms $C \to B^{\pm}$. The claim follows from the argument in \ref{c3}.~and the equation $\deg B = \deg B_1 = \deg B_2$.
\item We distinguish the cases of $S$ being a field and of $S \iso F \times F$. In the first case, since $Z \not \iso S$ are both quadratic extensions, the tensor product $Z \otimes S$ is a field as well. The homomorphism $\gamma \colon C \otimes \opp{B} \to \End_S(E)$ as in \ref{c3}.~takes values in $\End_{Z\otimes S} E$ and is $Z\otimes S$-linear if viewing $E$ as a $Z\otimes S$-vector space through $\gamma$. Hence $(\deg B)^2 = \dim_S E = 2 \dim_{Z \otimes S} E = 2 \deg C \deg B \cdot 2^{d(C\otimes \opp{B},1)}$. By the isomorphism $(C \otimes S) \otimes_{Z\otimes S} \opp{(B \otimes Z)} \iso (C \otimes S) \otimes_{Z\otimes S} (Z \otimes \opp{B}) \iso C \otimes \opp{B}$, the claim follows. In the second case, there exist homomorphisms $C \to B^{\pm} $. As above we conclude $\deg B^{\pm} = 2 n_{\pm} \deg C 2^{d(B^{\pm}\otimes Z,C)}$ for some $n_{\pm}\geq 1$, hence $\deg B = 2 n \deg C \lcm \left(2^{d(B^{+}\otimes Z,C)}, 2^{d(B^{-}\otimes Z,C)} \right) = n \deg C \cdot 2^{d(B\otimes Z,C \otimes S)+1}$, where $n \geq 1$.
\end{enumerate}
\end{proof}

\begin{theorem}[Degree of minimal compositions]
\label{main-th}
Let $P$ be an extended quadratic pair of degree $n>1$. Let $C = C(P)$ the associated Clifford algebra and $Z=Z(C)$ be its center. 
\begin{enumerate}
\item Compositions with unitary involutions: \par 
Let $S/F$ be a separable quadratic extension and let $c' \in \Brc(S)$ with $N_{S/F}(c')=1$ in $\Br(F)$. The degrees of minimal compositions are given as follows:
	\begin{enumerate}
		\item $n=2k+1$:
			\label{m1b} We have $\mcd(P, c',0) = 2^{k+d(c',C\otimes S)}$.
			For any composition $\extpaircomp$ of type $(c',0)$, the degree of $B$ is a multiple of $\mcd(P,c,t)$.
		\item $n=4k$:
			\label{m2b} If $Z \iso F \times F$, $C \iso C^{+} \times C^{-}$ then $\mcd(P,c',0) = 2^{2k-1+\min\left\{d(c',C^{+}\otimes S), d(c',C^{-}\otimes S)\right\}}$.\par
				If $Z$ is a field and $Z \not \iso S$ then $\mcd(P,\Brcl{D'},0) = 2^{2k+d(D'\otimes Z, C \otimes S)}$ and in general a multiple of $\mcd(P,\Brcl{D'},0)$.
		\item $n=4k+2$: 
			\label{m3b} If $Z \iso S$ then $\mcd(P,c',0) = 2^{2k + \min \left\{d(c',C),d(c',\opp{C}) \right\}}$. \par
			If $Z$ is a field and $Z \not \iso S$ then $\mcd(P,\Brcl{D'},0) = 2^{2k+1+d(D'\otimes Z, C \otimes S)}$ and in general a multiple of $\mcd(P,\Brcl{D'},0)$. \par
			%
	\end{enumerate}
\item Compositions with involutions of the first kind: \par
	Let $c \in \Br(F)$ with $c^2 = 1$ in $\Br(F)$, let $t \in \mu_2(F)$ and let $\epsilon = 0$ if the canonical involution of $C$ is of type $t$, 	otherwise $\epsilon =1$. Then the degrees of minimal compositions are given as
	follows:
	\begin{enumerate}
		\item $n = 2k+1$: 
			\label{m1a} We have $\mcd(P, c, t) = 2^{k + d(c,C)+ \delta}$, where $\delta = 1$ if $[C]=c$ and $\epsilon =1$,
			otherwise $\delta = 0$.
			For any composition $\extpaircomp$ of type $(c, t)$, the degree of $B$ is a multiple of $\mcd(P,c,t)$.
		\item $n=4k$:  
			\label{m2a} If $Z$ is a field, then $\mcd(P,[D],t) = 2^{2k+d(D \otimes Z, C)+ \delta}$ 
			and in general a multiple of the minimal degree; \\
			if $Z \iso F \times F$, $C \iso C^{+} \times C^{-}$, then the minimal composition degree is 
			$\mcd(P,c,t) = 2^{2k-1 + \min \left\{d(c, C^{+}), d(c,C^{-}) \right\} + \delta}$, 
			where in the two cases $\delta = 1$ if $[C] = [D\otimes Z]$ and $([C^{+}]=c~\text{or}~[C^{-}]=c)$, respectively, and
			$\epsilon = 1$, otherwise $\delta = 0$.
		\item $n=4k+2$:
			\label{m3a} We have $\mcd(P,[D],t) = 2^{2k+1+d(D \otimes Z, C)}$. In general $\deg B = n \cdot 2^{2k+d(D \otimes Z,
			C)}$ with $n\geq 2$ and $n$ even if $Z$ is a field.
	\end{enumerate}
\end{enumerate}
\end{theorem}

The proof of Theorem (\ref{main-th}) consists of 2 parts: existence and minimality. For both in addition to Lemma (\ref{cbound}) the following Lemma about the extension of involutions is needed. It can be found in \cite[p.~52]{boi} and is reproduced here, as far as necessary:

\begin{lemma}
\label{ext}
Let $B$ be a simple sub-algebra of a central simple algebra $A$ over a field $K$. Suppose, that $A$ and $B$ have involutions $\sigma$ and $\tau$ respectively with the same restriction to $K$. Then $A$ has an involution $\sigma'$ whose restriction to $B$ is $\tau$. \par
If $\sigma$ is of the first kind, the types $\sigma'$ and $\tau$ are related as follows:
\begin{itemize}
\item If $\Char K \neq 2$, then $\sigma'$ can be arbitrarily chosen of orthogonal or symplectic type, except if $\tau$ is of the first kind and $\deg C_A(B)$ is odd. In that case, every extension $\sigma'$ of $\tau$ is of the same type as $\tau$.
\item Suppose $\Char K = 2$: If $\tau$ is symplectic or unitary, then every extension $\sigma'$ of $\tau$ is symplectic.
\end{itemize}
\end{lemma}
For unitary involutions on Azumaya-algebras over $F \times F$ we use the following construction:
\begin{lemma}
\label{fxfext}
Let $R = F \times F$.
\begin{enumerate}
\item Let $B$ be $R$-Azumaya and sub-algebra of an Azumaya $R$-algebra $A$. Assume, that $A$ and $B$ possess unitary involutions $\sigma$ and $\tau$, respectively. Then $A$ has an involution $\sigma'$, which restricts to $\tau$ on $B$. 
\item Let $D$ be a central simple division-algebra with an involution of the first kind $d \mapsto \bar{d}$ and let $B$ be $R$-Azumaya of the Form $B = \End_{D \otimes R}(E)$ for a faithfully projective $D \otimes R$-module $E$. Let $\tau$ be a unitary involution on $B$. Then there exists an involution $\sigma'$ of the first kind on $\End_{D}(E) \supset \End_{D \otimes R}(E)$, which extends $\tau$ and $\sigma'$ can be arbitrarily chosen of orthogonal or symplectic type if $\Char F \neq 2$ and of symplectic type if $\Char F = 2$, respectively.
\end{enumerate}
\end{lemma}
\begin{proof}
\begin{enumerate}
\item Because of the unitary involutions $A$ and $B$ have constant rank. There exists a Azumaya $R$-algebra $C$ of constant rank and unitary involution $\rho$ with $A = B \otimes C$. The involution $\sigma' = \tau \otimes \rho$ extends $\tau$ to $A$.
\item Let $k=\deg B$. Let $\ast$ be the involution on $M_k(D)$ given by $(d_{ij})^{\ast} = (\overline{d_{ij}})^t$. We identify $\End_{D}(E)$ with $M_{2k}(D)$ and $(\End_{D \otimes R}(E), \tau)$ with $(M_k(D) \times M_k(D), \rho)$, $\rho(m_1,m_2) = (m_2^{\ast}, m_1^{\ast})$. Let $\sigma_{\pm}$ be the involution on $M_2(F)$ defined by 
$\sigma_{\pm} \left(\begin{array}{cc}
a & b \\
c & d
\end{array}
\right) =
\left(\begin{array}{cc}
d & \pm b \\
\pm c & a
\end{array}
\right).$
The involution $\sigma_{-}$ is symplectic, $\sigma_{+}$ is orthogonal if $\Char F \neq 2$. 
Define $\sigma'$ as the tensor product $\ast \otimes \sigma_{\pm}$ on $M_{2k}(D) = M_2(M_k(D))$, where $\sigma_{+}$ or $\sigma_{-}$ are chosen appropriately such that $\sigma'$ has the required type. By construction $\sigma'|_{M_k(D) \times M_k(D)} = \rho$.
\end{enumerate}
\end{proof}

\begin{proof}[Construction of Compositions]
For an extended quadratic pair $P$ we construct compositions with involutions of the first kind, i.e. of type $(P,c,t)$ with $c \in \Br(F)$, $t \in \mu_2(F)$ and compositions with unitary involutions, i.e. of type $(P,c',0)$ with $c' \in \Brc(S)$ for some separable quadratic extension $S/F$.
\begin{enumerate}
\item Compositions with unitary involutions:
\begin{enumerate}
\item $\deg P=2k+1$: There exists a homomorphism $C(P) \otimes S \to B$, where $B$ is of Brauer class $c'$ and degree $2^{k+d(c',C(P)\otimes S)}$. Since $B$ has a unitary involution and the involution $\usigma \otimes \iota$ on $C(P) \otimes S$ extending $\usigma$ is unitary, there exists a unitary involution $\tau$ on $B$ which extends the transport of $\usigma$. \par
\item $\deg P = 4k$: The involution on $C(P)$ restricts to the identity on $Z=Z(C(P))$. First consider the case $Z \iso F \times F$. The projections onto $C^{+}$ and $C^{-}$ yield two homomorphisms of algebras-with-involutions $(C(P),\usigma) \to (C^{\pm},\usigma^{\pm})$. Proceeding as above yields compositions of degree $2^{2k-1+d(c',C^{\pm}\otimes S)}$. The composition with smaller degree gives the asserted minimal degree. \par
Now assume $Z$ is a field. Consider a homomorphism $C \to C \otimes S \to \End_{D'\otimes Z}(E)$, where $\tilde{B} = \End_{D'\otimes Z}(E)$ is $Z \otimes S$-Azumaya of degree $2^{2k-1+d(D' \otimes Z, C \otimes S)}$. By Lemma \ref{ext} or \ref{fxfext} there exists an involution $\tilde{\tau}$ on $\tilde{B}$ which extends $\usigma \otimes \iota$, in particular $\tilde{\tau}|_{Z\otimes S} = \id_Z \otimes \iota_S$. Now if $S$ is a field, $S \not \iso Z$ Lemma \ref{ext} shows, that $\tilde{\tau}$ can be extended to an unitary involution on $\End_{D'}(E)$. If $S \iso F\times F$ we use an explicit construction: There exists  a central simple $F$-algebra $G$ and an isomorphism $(\tilde{B}, \tilde{\tau}) \iso (G\otimes Z \times \opp{G} \otimes Z, \rho)$ where $\rho$ is the involution $\rho(g_1 \otimes z_1, \opp{g_2} \otimes z_2) = (g_2 \otimes z_2, \opp{g_1} \otimes z_1)$. Clearly $\rho$ leaves $G \times \opp{G} \subset (G \times \opp{G}) \otimes Z$ invariant. By extending the identity on $Z$ to an involution of the first kind on $M_2(F)$ we get an involution $\tau$ on $M_2(G \times \opp{G}) \iso \End_{D'}(E)$ extending $\tilde{\tau}$.
\item $\deg P = 4k+2$: The involution on $C(P)$ restricts to the standard involution on $Z$. First consider the case $Z \iso S$. Observe that $(C(P),\usigma)$ and $(\opp{C(P)},\opp{\usigma})$ are isomorphic as $F$-algebras-with-involutions. There exist homomorphisms $C(P)\to B$ and $C(P) \to B'$ where $\Brcl{B} = \Brcl{B'} = c'$ and $\deg B = 2^{2k+d(c',C(P))}$, $\deg B' = 2^{2k+d(c',\opp{C(P)})}$. The unitary involutions on $C(P)$ and $\opp{C(P)}$ can be extended to unitary involutions on $B$ and $B'$, respectively. The composition with smaller degree gives the asserted minimal degree. \par
Now assume $Z$ is a field and $Z \not \iso S$. Let $\tilde{B}= \End_{D'\otimes Z}(E)$ be of minimal degree such that there exists a homomorphism $C \otimes S \to \tilde{B}$. By Lemma \ref{dbound} $\deg \tilde{B} = 2^{2k+d(D' \otimes Z,C \otimes S)}$. The involution $\usigma \otimes \iota$ on $C(P) \otimes S$ extends to an involution $\tilde{\tau}$ on $\tilde{B}$ with $\tilde{\tau}|_{Z\otimes S} = \iota_Z \otimes \iota_S$. If $S$ is a field then by Lemma \ref{ext} $\tilde{\tau}$ extends to a unitary involution $\tau$ on $B=\End_{D'}(E)$. If $S \iso F \times F$ we use the following construction: There exists a central simple $F$-algebra $G$ and an isomorphism of $F$-algebras $(\tilde{B}, \tilde{\tau}) \iso (G \otimes Z, \times \opp{G} \otimes Z, \rho)$ where $\rho(g_1 \otimes z_1, \opp{g_2} \otimes z_2) = (g_2 \otimes \bar{z_2}, \opp{g_1} \otimes z_1)$. Extending the standard involution on $Z$ to an involution of the first kind on $M_2(F)$ yields an involution $\tau$ on $M_2(G \times \opp{G}) \iso \End_{D'}(E)$ extending $\tilde{\tau}$.

\end{enumerate}
\item Compositions with involutions of the first kind:
\begin{enumerate}
\item $\deg P=2k+1$, there exists a homomorphism $C(P) \to B$, where $B$ is of Brauer class $c$ and degree $2^{k+d(c,C(P))}$. If $[C(P)]\neq c$, then $\deg C_B(C(P))$ is of the form $2^{l}, l\geq 1$ since there exist an involution of the first kind on $C_B(C(P))$, in particular even. Thus, if $[C(P)]\neq c$ or if $\usigma$ is already of the required type, the involution $\usigma$ on $C(P)$ can be extended to an involution of type $t$ on $B$. Otherwise, embedding $B$ into $M_2(F) \otimes B$ with symplectic involution on $M_2(F)$ yields a composition of type $t$ and degree $2^{k+d(c,C(P))+1}$. In both cases we get compositions of the required type and degree. \par
\item $\deg P=4k$: If $Z(C(P))\iso F\times F$, projection to $C^{\pm}$ and arguing as for odd degree of $P$ yields a composition of the required degree and type. If $Z(C(P))=K$ is a field, there exists a homomorphism $C(P) \to \tilde{B} = \End_{D\otimes K}(E) \subset \End_D(E)$, where $E$ is a faithfully projective $D\otimes K$-module and $\deg \tilde{B} = 2^{2k+d(D \otimes K, C(P))}$. Let $B=\End_D(E)$. If $[C(P)]\neq [D]$, then $\deg C_B(C(P)) = \deg C_{\tilde{B}}(C(P))$ is even. Thus, if $[C(P)] \neq [D]$ or if $\usigma$ is already of the required type, the involution $\usigma$ on $C(P)$ can be extended to an involution of type $t$ on $B$. Otherwise, after embedding $B$ in $M_2(B)$ the required type can be obtained. \par
\item $\deg P=4k+2$: There exists a homomorphism $C(P) \to \tilde{B} = \End_{D\otimes Z}(E)$, where $E$ is a faithfully projective $D\otimes Z$-module of constant rank and $\deg \tilde{B} = 2^{2k+d(D \otimes Z, C(P))}$. The unitary involution $\usigma$ on $C(P)$ can be extended to a unitary involution on $\tilde{B}$. Embedding $\tilde{B} = \End_{D\otimes Z}(E)$ in $B=\End_D(E)$ and extending the involution of $\tilde{B}$ to a involution of type $t$ on $B$ yields a composition of required type and degree $2^{2k+d(D \otimes Z, C(P))+1}$.
\end{enumerate}
\end{enumerate}
\end{proof}
\begin{proof}[Minimality of the Constructed Compositions]
\hspace*{\fill}
\begin{enumerate}
\item Compositions with unitary involutions:
\begin{enumerate}
\item $\deg P=2k+1$: Let $\alpha \colon C(P) \to B$ be a composition homomorphism. Minimality as well as the assertion, that $\deg B$ is a multiple of $2^{k+d(C(P),B)}$ follow from Lemma \ref{cbound},\ref{c1}.

\item $\deg P=4k$: Let $\alpha \colon C(P) \to B$ be a composition homomorphism. First consider the case $Z=Z(C(P))\iso F\times F$. If $\alpha$ is non-injective, then $C(P)=C^{+}\times C^{-}$ and $\alpha$ factors through either $C^{+}$ or $C^{-}$. Lemma \ref{cbound}, \ref{c2} gives $\deg B \geq 2^{2k-1+d(C^{+}\otimes S,c')}$ or $\deg B \geq 2^{2k-1+d(C^{-}\otimes S, c')}$, respectively, showing the claim. If $S=Z(B)$ is a field, then by \ref{cbound},\ref{c3} there exist $n_1,n_2 \in \mathbb{N}$, $n_1+n_2 \geq 1$ with $\deg B = 2^{2k-1} \big(n_1 2^{d(B,C^{+}\otimes S)} + n_1 2^{d(B,C^{-}\otimes S)} \big) \geq 2^{2k-1 + d(B,C^{\pm}\otimes S)}$. If $S\iso F\times F$, then by \ref{cbound},\ref{c3} there exist $n_1, n_2 \in \mathbb{N}$ with $n_1+n_2\geq 1$ and $\deg B = \deg C\big(n_1 2^{d(B^{+},C^{+})}+n_2 2^{d(B^{+},C^{-})}\big) \geq \deg C 2^{d(B^{+},C^{\pm})}$. Since there exists a unitary involution on $B = B^{+} \times B^{-}$ we have $B^{-} \iso \opp{{B^{+}}}$ and moreover $\opp{{C^{\pm}}} \iso C^{\pm}$, hence $d(B^{-},C^{\pm})=d(B^{+},\opp{{C^{\pm}}})$. Thus $2^{d(B,C^{\pm}\otimes S)} = \lcm\big(2^{d(B^{+},C^{\pm})}, 2^{d(B^{-},C^{\pm})}\big) = 2^{d(B^{+},C^{\pm})}$ shows the assertion. \par
Now assume that $Z$ is a field, $Z \not \iso S$. In that case minimality follows straightly from \ref{cbound},6.

\item $\deg P = 4k+2$: Let $\alpha \colon C(P) \to B$ be a composition homomorphism. First consider the case $Z\iso S$. Then by \ref{cbound},\ref{c4} there exist $n_1, n_2 \in \mathbb{N}$ with $n_1+n_2\geq 1$ such that 
$\deg B = 2^{2k} \big(n_1 2^{d(B,C(P))} + n_2 2^{d(\inv{B}, C(P))} \big) \geq 2^{2k + \min \{d(B,C(P)), d(\inv{B}, C(P)) \} }$ and since $d(\inv{B}, C(P)) = d(\opp{B}, C(P)) = d(B,\opp{C(P)})$ the claim follows. \par
Now assume that $Z$ is a field and $Z \not \iso S$. In that case minimality follows straightly from \ref{cbound},6.
\end{enumerate}

\item Compositions with involutions of the first kind:
\begin{enumerate}
\item $\deg P=2k+1$: Lemma \ref{dbound} gives $\deg B = n 2^{k+d(c,C(P))}$. If $[C(P)] \neq c$ or the involution $\usigma$ on $C(P)$ is of the required type, the claim follows. Otherwise $\Char F \neq 2$ ($\deg P > 1$) and we cannot have odd $n$, since $\deg C_B(C(P)) = n$ and by Lemma \ref{ext} the involution $\tau$ extending $\usigma$ would then have to be of the same type. Thus, $\deg P = m 2^{k+d(c,C(P))+1}$ with $m=\frac{n}{2} \in \mathbb{N}$.
\item $\deg P=4k$: First let $Z=Z(C(P))$ be a field. By \ref{cbound},\ref{c6} there exists $n\in \mathbb{N}$ with $\deg B = n 2^{2k+d(B\otimes Z, C(P))}$. If $[C(P)] \neq [B\otimes Z]$ or if the involution $\usigma$ on $C(P)$ is of the required type, the claim follows. Otherwise by Lemma \ref{ext} $\deg C_B(C(P))=n$ must be even and thus $\deg B = m 2^{2k+1}$ with $m = \frac{n}{2} \in \mathbb{N}$. \par
Let now $Z \iso F \times F$. By \ref{cbound},\ref{c3} there exist $n_1, n_2 \in \mathbb{N}$, $n_1+n_2\geq 1$ with $\deg B = 2^{2k-1} \big(n_1 2^{d(B,C^{+})} + n_2 2^{d(B,C^{-})} \big)$. If $[C^{+}] \neq [B]$, $[C^{-}] \neq [B]$ or if the involution $\usigma$ on $C(P)$ is of the required type, the claim follows. Otherwise, since $\usigma$ is not of the right type $\deg B \neq \deg C = 2^{2k-1}$ and thus $\deg B \geq 2^{2k}$, showing minimality.
\item $\deg P = 4k+2$: In that case $\usigma$ is unitary and every composition $\extpaircomp$ is injective. Otherwise there would exist a nontrivial central idempotent $e \in \ker \alpha$ and the equation $0 = \tau(\alpha(e)) = \alpha(\usigma(e)) = \alpha(1-e) = 1$ would lead to a contradiction. If $Z=Z(C(P))$ is a field, then by \ref{cbound},\ref{c6} there exists $n\in \mathbb{N}$ with $\deg B = n 2^{2k+1+d(B\otimes Z,C(P))}$ showing the claim. If $Z\iso F \times F$, Lemma \ref{cbound},\ref{c3} yields $n_1, n_2 \in \mathbb{N}$ with $\deg B = \big( n_1 2^{2k+d(B, C^{+})} + n_2 2^{2k+d(B, C^{-})} \big)$ which is equal to $(n_1+n_2)2^{2k+d(B\otimes Z,C(P))}$ since $2^{d(B\otimes Z, C(P))} = \lcm(2^{d(B,C^{+})},2^{d(B,C^{-})})$ and furthermore $d(B,C^{+}) = d(\opp{B}, \opp{{C^{+}}}) = d(B,C^{-})$.  
\end{enumerate}

\end{enumerate}
\end{proof}

\begin{example}
For $\deg P$ even we have computed the minimal composition degree for compositions with unitary involutions only under restriction to the center $Z$ of $C(P)$, namely we have excluded the case $Z \iso S$ being a field if $\deg P \equiv 0 \pod 4$ and $Z \iso F \times F$ where $S$ is a field if $\deg P \equiv 2 \pod 4$, respectively. Also in those cases Lemma \ref{cbound} shows how to reach the minimal degree of an algebra $B$ of given class $c \in \Brc(S)$ such that there exist a homomorphism $C(P) \to B$. The problem however is, that the canonical involution on $C(P)$ cannot be extended to $B$ in general. Consider e.g.~compositions of type $(c' = \Brcl{D'}, 0)$ for quadratic pairs of degree $4k$, where $Z=Z(C(P))$ and $S=Z(D')$ are isomorphic fields and $D'$ is some division $S$-algebra. Lemma \ref{cbound},3 gives $\deg B = 2^{2k-1} \big(n_1 2^{d(B,C)} + n_2 2^{d(\inv{B},C)}\big)$ for some $n_1,n_2 \in \mathbb{N}$ with $n_1+n_2 \geq 1$. A composition $\extpaircomp$ cannot be $Z$-linear, since the involutions on $C(P)$ and $B$ are not of the same kind. Thus it follows $n_1 \geq 1$ and $n_2 \geq 1$ as can be seen from the proof of \ref{cbound},3. It is in fact possible to construct algebra-homomorphisms $C(P) \to B$ where $\deg B$ is of the above form. For that we consider $C(P) \hookrightarrow C(P) \otimes S \iso C(P) \times \inv{C(P)}$. Choose $B_1 = \End_D'(E_1)$ and $B_2 = \End_D'(E_2)$ of minimal degree such that there exist homomorphisms $C(P) \to B_1$ and $\inv{C(P)} \to B_2$. That yields a homomorphism $C(P) \to \End_{D'}(E_1) \times \End_{D'}(E_2) \hookrightarrow B=\End_D'(E_1\oplus E_2)$ or more generally $C(P) \to B=\End_{D'}(n_1 E_1 \oplus n_2 E_2)$ for $n_1 \geq 1$, $n_2 \geq 1$. Still it is not possible to extend the canonical involution on $C(P)$, since under the above homomorphism the center of $C(P)$ maps onto the center of $B$. It becomes all much easier if we take only quadratic pairs over trivial algebras, i.e~quadratic forms. Then the Clifford algebra $(C(P),\usigma) = (C_0(V,q),\tau_0)$ can be embedded into the full Clifford algebra $C(V,q)$ (with canonical or the standard involution) and the same construction like for quadratic pairs of odd degree yields a composition of degree $2^{2k + d(C(V,q)\otimes S, c')} = 2^{2k+d(C_0(V,q),c')}$. By \ref{cbound},3 (where $n_1 n_2 = 0$ is excluded) the degree of that composition is minimal.
\end{example}

In Theorem (\ref{main-th}) a bound for $\deg B$ which is independent of the structure of the Clifford algebra $C(P)$ can be read off. The Theorem indicates as well, when that bound is reached:
\begin{corollary}
\label{main-corollary}
Let $P$ be an extended quadratic pair of degree $n>1$ and let $\extpaircomp$ a composition.
\begin{enumerate}
	\item Compositions with unitary involutions:
	\begin{enumerate}
		\item $n = 2k+1$: We have $\deg B \geq 2^k$ with equality if and only if $\Brcl{C(P) \otimes Z(B)} = c'$.
		\item $n = 4k$: We have $\deg B \geq 2^{2k-1}$ with equality if and only if $Z\iso F\times F$ and 
			\big($\Brcl{C^{+}\otimes Z(B)} = c'$ or $\Brcl{C^{-}\otimes Z(B)} = c'$\big).
		\item $n=4k+2$: We have $\deg B \geq 2^{2k}$ with equality if and only if $Z\iso S$ and $\Brcl{C(P)}=c'$.
	\end{enumerate}
	\item Compositions with involutions of the first kind: Let $t \in \mu_2(F)$ be the type of $\tau$.
	\begin{enumerate} 
		\item $n=2k+1$: We have $\deg B \geq 2^{k + \delta}$, where $\delta = 1$ if $\usigma$ is of type $t$,
		otherwise $\delta = 0$. \\
		In case of $\delta=0$ equality holds if and only if $[C(P)]=c$.
		In case of $\delta=1$ equality holds if and only if $[C(P) \otimes Q]=c$ for a quaternion algebra $Q$.
		\item $n=4k$: We have $\deg B \geq 2^{2k-1 + \delta}$, where $\delta = 1$ if $\usigma$ is of type $t$, otherwise $\delta = 0$. \\
		In case of $\delta = 0$ equality holds if and only if $Z\iso F\times F$ and \big($[C^{+}]=c$ or $[C^{-}]=c$\big).
		In case of $\delta = 1$ equality holds if and only if \big($[C^{+}\otimes Q]=c$ or $[C^{-}\otimes Q]=c$\big) for a quaternion algebra $Q$.
		\item $n=4k+2$: We have $\deg B \geq 2^{2k+1}$ with equality if and only if $\Brcl{C(P)} = \Brcl{D\otimes Z}$.
	\end{enumerate}
\end{enumerate}
\end{corollary}
\begin{proof}
Everything follows straightly from Theorem (\ref{main-th}), except the assertions about the center of the Clifford algebra in \ref{m2b}) and \ref{m3b}). In both cases, the degree of a minimal $B$ coincides with the degree of the Clifford algebra. Since a composition of an extended quadratic pair of degree $\deg P = 4k+2$ must be injective, both algebras are isomorphic and in particular their centers are isomorphic. On the other hand if $\deg P = 4k$, then a minimal composition cannot be an isomorphism, because the involutions on the centers are not of the same kind. Hence, the kernel is nontrivial and in particular, $Z \iso F \times F$.
\end{proof}
The reader may apply Corollary \ref{main-corollary} to compare with the results of \cite{zueger2} in the special cases of (anti-)symmetric compositions and hermitian compositions over quadratic and quaternion algebras. We are doing that only for quadratic spaces $(V,q)$ of dimension $4k$ and compositions with symplectic involutions. That corresponds to compositions with antisymmetric forms (resp.~hermitian forms over a quaternion algebra with respect to the standard involution), see \cite[Sections 4, 6]{zueger2}. According to the corollary, the minimal degree of $B$ for compositions of the form $\alpha \colon \big(C(\End_F(V), \sigma_q, f_q), \usigma_q\big) \to (B,\tau)$, $\tau$ symplectic is given by $2^{2k}$ if $\dim V \equiv 0 \pod 8$ and $\Char F \neq 2$, and $2^{2k-1}$ if $\dim V \equiv 4 \pod 8$ or $\Char F =2$, respectively. Moreover equality holds if and only if $Z\iso F\times F$ and $[C^{\pm}]=[B]$ (resp.~$[C^{\pm}] = [B \otimes Q']$ for a  quaternion algebra $Q'$). In both cases $[C^{\pm}]=[C(V,q)]$. It remains to remark, that $\deg B =\dim E$ if $B = \End_F(E)$ and $\deg B = \frac{1}{2} \dim E$ if $B = \End_Q(E)$.

\section*{Acknowledgments}
I wish to express my gratitude to Prof.~Max-Albert Knus, adviser of my diploma thesis, for a lot of support and inspiring discussions.

\bibliography{reference}
\bibliographystyle{alpha}

\end{document}